\pdfoutput=1
\RequirePackage{ifpdf}
\ifpdf % We are running pdfTeX in pdf mode
\documentclass[pdftex]{sigma}
\else
\documentclass{sigma}
\fi

\begin{document}

\allowdisplaybreaks

\newcommand{\arXivNumber}{1406.4652}

\renewcommand{\PaperNumber}{009}

\FirstPageHeading

\ShortArticleName{Bipolar Lawson Tau-Surfaces and Generalized Lawson Tau-Surfaces}

\ArticleName{Bipolar Lawson Tau-Surfaces\\ and Generalized Lawson Tau-Surfaces}

\Author{Broderick {CAUSLEY}}

\AuthorNameForHeading{B.~Causley}

\Address{Department of Mathematics and Statistics, McGill University, Burnside Hall,\\ 805 Sherbrooke Street West,
Montreal, QC H3A 0B9, Canada}
\Email{\href{mailto:broderick.causley@mail.mcgill.ca}{broderick.causley@mail.mcgill.ca}}

\ArticleDates{Received November 27, 2015, in f\/inal form January 21, 2016; Published online January 25, 2016}

\Abstract{Recently Penskoi [\textit{J.~Geom.~Anal.}~\textbf{25} (2015), 2645--2666, arXiv:1308.1628]  ge\-ne\-ra\-li\-zed the well known two-parametric family of Lawson tau-surfaces $\tau_{r,m}$ minimally immersed in spheres to a three-parametric family $T_{a,b,c}$ of tori and Klein bottles minimally immersed in spheres. It was remarked that this family includes surfaces carrying all extremal metrics for the f\/irst non-trivial eigenvalue of the Laplace--Beltrami operator on the torus and on the Klein bottle: the Clif\/ford torus, the equilateral torus and surprisingly the bipolar Lawson Klein bottle $\tilde{\tau}_{3,1}$. In the present paper we show in Theo\-rem~\ref{theorem1} that this three-parametric family $T_{a,b,c}$ includes in fact all bipolar Lawson tau-surfaces $\tilde{\tau}_{r,m}$. In Theo\-rem~\ref{theorem3} we show that no metric on generalized Lawson surfaces is maximal except for~$\tilde{\tau}_{3,1}$ and the equilateral torus.}

\Keywords{bipolar surface; Lawson tau-surface; minimal surface; extremal metric}

\Classification{58E11; 58J50; 49Q05; 35P15}

\section{Introduction}

Minimal surfaces in the Euclidean space~$\mathbb{R}^3$ is one of the most classical subjects of dif\/ferential geometry. Surprisingly, interest to minimal surfaces in spheres~$\mathbb{S}^n$ appeared much later and a~serious investigation in this direction started only in the second half of the 20th century. One important paper in the early period of this investigation was the famous paper~\cite{Lawson1970} by Lawson, where several examples of immersed minimal surfaces in~$\mathbb{S}^3$ were constructed, including examples of surfaces of any genus except the projective plane. These examples include a family $\tau_{m,n}$ of minimal tori and Klein bottles.

\begin{definition}
A Lawson tau-surface $\tau_{m,n}\looparrowright\mathbb{S}^3$ is def\/ined as the image of the doubly-periodic immersion $\Psi_{m,n}\colon \mathbb{R}^2\looparrowright\mathbb{S}^3\subset\mathbb{R}^4$ given by the explicit formula
\begin{gather}\label{immersion}
\Psi_{m,n}(x,y)=(\cos(mx)\cos y, \sin(mx)\cos y, \cos(nx)\sin y,\sin(nx)\sin y).
\end{gather}
\end{definition}

For each unordered pair of positive integers $(m,n)$ with $(m,n)=1$, the surface $\tau_{m,n}$ is a~distinct compact minimal surface in $\mathbb{S}^3$. Let us assume $(m,n)=1$. If both integers~$m$ and~$n$ are odd then $\tau_{m,n}$ is a~torus (we call it Lawson torus). If one of integers~$m$ and~$n$ is even then~$\tau_{m,n}$ is a~Klein bottle (we call it a~Lawson Klein bottle). The torus~$\tau_{1,1}$ is the Clif\/ford torus.

Many ef\/forts were later applied to f\/ind minimal surfaces in spheres. There are several constructions of paticular types of minimal surfaces in spheres.  For example, the Hsiang--Lawson approach~\cite{Hsiang-Lawson1971} constructs surfaces with additional symmetry. However it is only minimal tori that are completely described. This description is given by the theory of integrable systems using algebraic geometry. This was done in several papers by many authors starting from Hitchin's paper~\cite{Hitchin1990} dealing with a particular case of~$\mathbb{S}^3$ and f\/inishing with Burstall's paper~\cite{Burstall1995} dealing with the general case $\mathbb{S}^n$. In fact this investigation was a part of more general investigation of harmonic maps from tori into symmetric spaces. A recent paper by Carberry~\cite{Carberry2012} contains a~review of the current state of this approach.

We should remark however that {\em only a very limited number of explicitly parametrized minimal surfaces in spheres is known}. Unfortunately, most approaches to minimal surfaces in spheres give only quite implicit descriptions of minimal surfaces. This pertains to the Hsiang--Lawson construction and description of minimal tori using integrable systems mentioned above. This makes the question of f\/inding {\em explicitly parametrized} minimal surfaces in spheres interesting.

Let us mention here several known results about explicitly parametrized minimal surfaces in spheres. We mentioned already Lawson surfaces $\tau_{m,n}$ given by the explicit parametrization~\eqref{immersion}. Another example is bipolar Lawson surfaces. Starting from an immersed surface $I\colon \Sigma\looparrowright\mathbb{S}^3\subset\mathbb{R}^4$ one can construct a new surface $\tilde{\Sigma}$ called bipolar to $\Sigma$ and def\/ined by the immersion $I\wedge I^*\colon \Sigma\looparrowright\mathbb{R}^4\wedge\mathbb{R}^4\cong\mathbb{R}^6$, where $I^*$ is a Gaussian map, i.e., $I^*(p)$ is a unitary normal vector to $\Sigma$ in $\mathbb{S}^3$ at the point~$p$. It turns out that the image is in the sphere $\mathbb{S}^5\subset\mathbb{R}^6$, and if the surface~$\Sigma$ is minimal in $\mathbb{S}^3$, then its bipolar surface $\tilde{\Sigma}$ is minimal in~$\mathbb{S}^5$ (see~\cite{Lawson1970}). As was shown by Lawson, bipolar surfaces $\tilde{\tau}_{m,n}$ are in fact lying in an equatorial sphere $\mathbb{S}^4\subset\mathbb{S}^5$. Details are provided below in Section~\ref{section2}.

Another example is given by a construction by Mironov~\cite{Mironov2004,Mironov2010} of Hamiltonian-minimal Lagran\-gian embeddings in~$\mathbb{C}^N$ based on intersections of real quadrics of a special type. As was remarked by Karpukhin in~\cite{kar2}, as a by-product of this construction, one obtains minimal surfaces in spheres.

Recently Penskoi constructed a generalization of Lawson surfaces~$\tau_{m,n}$ in~\cite{penskoi1}. This gene\-ra\-li\-zation is given by an explicit formula. In complex form, it is written as
\begin{gather}
\left(\sqrt{\frac{b^2+c^2-a^2}{2(c^2-a^2)}}e^{iax}\sin y,
\sqrt{\frac{a^2+c^2-b^2}{2(c^2-b^2)}}e^{ibx}\cos y,\right.\notag\\
\left.\qquad{} \sqrt{\frac{a^2+b^2-c^2}{2(b^2-c^2)}}e^{icx}%
\sqrt{1-\frac{b^2-a^2}{c^2-a^2}\sin^2y}\right)
\in\mathbb{S}^5\subset\mathbb{C}^3\cong\mathbb{R}^6.\label{Tabcexplicit}
\end{gather}
It is proven in~\cite{penskoi1} that
\begin{itemize}\itemsep=0pt
\item if either $a$, $b$, $c$ are integers and $|c|>\sqrt{a^2+b^2}$ or
\item if $a$, $b$ are nonzero integers and $|c|=\sqrt{a^2+b^2}$,
\end{itemize}
then the parametrization~\eqref{Tabcexplicit} def\/ines a minimal torus or a minimal Klein bottle in $\mathbb{S}^5$, denoted by $T_{a,b,c}$. Moreover, if $|c|=\sqrt{a^2+b^2}$, then $T_{a,b,c}$ coincides with the Lawson surface~$\tau_{a,b}$. This follows immediately by composing~\eqref{immersion} and~\eqref{Tabcexplicit}. Thus, this three-parametric family of minimal tori or Klein bottles in fact generalizes the famous two-parametric Lawson family $\tau_{m,n}$.

Penskoi's construction of generalized Lawson tori was deeply motivated by the relation between minimal surfaces in spheres and spectral geometry. We will expose this relation below, but let us also remark at this moment that it was observed in~\cite{penskoi1} that bipolar Lawson surfa\-ce~$\tilde{\tau}_{3,1}$ is isometric to $T_{1,0,2}$. The proof of this fact was indirect and was based on the uniqueness of the extremal metric for the f\/irst eigenvalue of the Laplace--Beltrami operator~$\Delta$ on the Klein bottle, proved in~\cite{sgj1}.

This observation is quite intriguing and leads us into an investigation of the following natural question: is it true that not only is the two-parametric family of Lawson tau-surfaces~$\tau_{m,n}$ a~subfamily of the three-parametric family~$T_{a,b,c}$ of generalized Lawson surfaces, but is the family of bipolar Lawson surfa\-ces~$\tilde{\tau}_{m,n}$ also a subfamily of the three-parametric family~$T_{a,b,c}$? The list of known explicitly parametrized minimal surfaces in spheres is quite short and this makes it interesting to understand relations among dif\/ferent known families of explicitly parametrized minimal surfaces in spheres.

The f\/irst result of the present paper is an af\/f\/irmative answer to this question.

\begin{theorem}\label{theorem1}\quad
\begin{enumerate} \itemsep=0pt
  \item[$1.$] If $rm\equiv 0\!\pmod{2}$ then the bipolar Lawson torus $\tilde{\tau}_{r,m}$ is isometric to
  the surface~$T_{a,b,c}$, where $a=r-m$, $b=0$, $c=r+m$.
  \item[$2.$] If $rm\equiv 1\!\pmod{4}$ then the bipolar Lawson torus $\tilde{\tau}_{r,m}$ is isometric to
  the surface~$T_{a,b,c}$, where $a=\frac{r-m}{2}$, $b=0$, $c=\frac{r+m}{2}$.
  \item[$3.$] If $rm\equiv 3\!\pmod{4}$ then the bipolar Lawson Klein bottle $\tilde{\tau}_{r,m}$ is isometric to the surfa\-ce~$T_{a,b,c}$, where $a=\frac{r-m}{2}$, $b=0$, $c=\frac{r+m}{2}$.
\end{enumerate}
\end{theorem}

Let us now explain the relation between minimal surfaces in spheres and spectral geometry in more detail.

Let $M$ be a closed surface and $g$ be a Riemannian metric on $M$. Let us consider the associated Laplace--Beltrami operator $\Delta \colon C^{\infty}(M)\rightarrow C^{\infty}(M)$,
\begin{gather*}
\Delta f=-\frac{1}{\sqrt{|g|}}\frac{\partial}{\partial x^{i}}\left(\sqrt{|g|}g^{ij}\frac{\partial f}{\partial x^{j}}\right).
\end{gather*}

The spectrum of $\Delta$ is non-negative and consists only of eigenvalues, where each eigenvalue has a f\/inite multiplicity and the associated eigenfunctions are smooth. Denote the eigenvalues of~$\Delta$ by
\begin{gather*}
0=\lambda_{0}(M,g)<\lambda_{1}(M,g)\le \lambda_{2}(M,g)\le \lambda_{3}(M,g)\le \cdots,
\end{gather*}
where eigenvalues are written with multiplicities.

Let us f\/ix the surface $M$ and consider $\Lambda_{i}(M,g)$ as a functional $g\mapsto\Lambda_{i}(M,g)$ on the space of all Riemannian metrics on~$M$. The eigenvalues possess the following rescaling property
\begin{gather*}
\forall \, t>0,\qquad \lambda_{i}(M,tg)=\frac{\lambda_{i}(M,g)}{t}.
\end{gather*}

To get scale-invariant functionals on the space of Riemannian metrics one has to normalize the eigenvalue functionals. It is most natural to normalize the functionals by multiplying by the area
\begin{gather*}
\Lambda_{i}(M,g)=\lambda_{i}(M,g)\operatorname{Area}(M,g).
\end{gather*}
The functionals $\Lambda_{i}(M,g)$ are invariant under the rescaling transformation $g\mapsto tg$.

If we consider the functional $\Lambda_{i}(M,g)$ over the space of Riemannian metrics~$g$ with a f\/ixed surface~$M$, the question about the value of supremum $\sup\Lambda_{i}(M,g)$ is interesting. It is a very dif\/f\/icult question with a limited number of known results. It follows from Yang and Yau~\cite{yy1} and Korevaar~\cite{k1} that this supremum is f\/inite.

\begin{definition}
A metric $g_{0}$ on a f\/ixed surface $M$ is called maximal for the functional $\Lambda_{i}(M,g)$ if
\begin{gather*}
\sup \Lambda_{i}(M,g)=\Lambda_{i}(M,g_{0}),
\end{gather*}
where the supremum is taken over the space of Riemannian metrics~$g$ on the f\/ixed surface~$M$.
\end{definition}

Currently, we only know maximal metrics for $\Lambda_{1}(\mathbb{S}^2,g)$, $\Lambda_{1}(\mathbb{R}P^2,g)$, $\Lambda_{1}(\mathbb{T}^2,g)$, $\Lambda_{2}(\mathbb{S}^2,g)$, and $\Lambda_{3}(\mathbb{S}^2,g)$. For more details please see the recent survey by Penskoi~\cite{penskoireview} and the recent paper by Nadirashvili and Sire~\cite{n2}.

A problem in the study of $\Lambda_{i}(M,g)$-maximal metrics is that the functional $\Lambda_{i}(M,g)$ depends continuously on metric $g$ but is not dif\/ferentiable. However, for any analytic deformation $g_{t}$, the left and right derivatives of the functional $\Lambda_{i}(M,g)$ with respect to \textit{t} exist (see Berger~\cite{b1}, Bando and Urakawa~\cite{bu1}, El Souf\/i and Ilias~\cite{si2}).

\begin{definition}[\cite{si1,si2, n1}] A Riemannian metric $g_{0}$ on a closed surface $M$ is called an extremal metric for the functional $\Lambda_{i}(M, g)$ if for any analytic deformation $g_{t}$ the following inequality holds
\begin{gather*}
\frac{d}{dt}\Lambda_{i}(M,g_{t})\bigg|_{t=0+}\cdot\frac{d}{dt}\Lambda_{i}(M,g_{t})\bigg|_{t=0-}\le 0.
\end{gather*}
\end{definition}

The mentioned metrics above are extremal since they are global maxima. However, extremal metrics are not necessarily maximal. For example, El Souf\/i and Ilias proved in \cite{si1} that the only extremal metric for $\Lambda_{1}(\mathbb{T}^2,g)$ dif\/ferent from the maximal one is the metric on the Clif\/ford torus.

Jakobson, Nadirashvili and Polterovich proved in \cite{jnp1} that the metric on the Klein bottle realized as the bipolar Lawson surface $\tilde{\tau}_{3,1}$ is extremal for $\Lambda_{1}(\mathbb{KL},g)$. Using this result El Souf\/i, Giacomini and Jazar proved in \cite{sgj1} that this metric is the unique extremal metric.

The following extremal metrics on families of tori and Klein bottles were investigated recently: Lapointe investigated metrics on bipolar Lawson surfaces $\tilde{\tau}_{r,m}\looparrowright\mathbb{S}^{4}$ in~\cite{lapointe}, these surfaces are described below in Section~\ref{section2}; Penskoi investigated extremal metrics on Lawson surfaces $\tau_{m,n}\looparrowright\mathbb{S}^{3}$ and on Otsuki tori $O_{\frac{p}{q}}\looparrowright\mathbb{S}^{3}$ in~\cite{penskoi2} and~\cite{penskoi3} respectively; Karpukhin investigated metrics on bipolar Otsuki tori $\tilde{O}_{\frac{p}{q}}\looparrowright\mathbb{S}^{4}$ and on a family of tori $M_{m,n}\looparrowright\mathbb{S}^{5}$ in~\cite{kar1} and~\cite{kar2} respectively; and Karpukhin proved that the metrics on $\tau_{m,n}$, $\tilde{\tau}_{r,m}$, $O_{\frac{p}{q}}$, $\tilde{O}_{\frac{p}{q}}$, and $M_{m,n}$ are not maximal except metrics on $M_{1,1}$ (the equilateral torus) and $\tilde{\tau}_{3,1}$ in~\cite{kar3}. Here $\looparrowright$ denotes an immersion.

As it was mentioned above, in~\cite{penskoi1}, Penskoi introduced the new three-parametric family~$T_{a,b,c}$ of minimal surfaces in spheres, generalizing Lawson tau-surfaces, and investigated their extremal spectral properties.

\begin{theorem}[Penskoi,~\cite{penskoi1}]\label{theorem2}  Let $F_{a,b,c}\colon \mathbb{R}^2\rightarrow\mathbb{S}^{5}\subset\mathbb{R}^{6}$ be a~three-parametric doubly-periodic immersion of the plane to the $5$-dimensional sphere of radius~$1$ defined by the formula~\eqref{Tabcexplicit}, where
\begin{itemize}\itemsep=0pt
\item[$a)$] either $a$, $b$, $c$ are integers and $|c|>\sqrt{a^2+b^2}$,
\item[$b)$] or $a$, $b$ are nonzero integers and $|c|=\sqrt{a^2+b^2}$.
\end{itemize}

 Then the following statements hold:
\begin{enumerate}\itemsep=0pt
\item[$1)$] The image $T_{a,b,c}=F_{a,b,c}(\mathbb{R}^2)$ is a minimal compact surface in the $5$-dimensional sphere~$\mathbb{S}^5$.

\item[$2)$]  The case $b)$ corresponds to Lawson tau-surfaces $\tau_{a,b}\cong T_{a,b,\sqrt{a^2+b^2}}$. Distinct Lawson tau-surfaces correspond to unordered pairs $a,b\geqslant1$ such that $(a,b)=1$. The surface $T_{a,b,\sqrt{a^2+b^2}}$ is a Lawson torus $\tau_{a,b}$ if $a$ and $b$ are odd and $T_{a,b,\sqrt{a^2+b^2}}$ is a Lawson Klein bottle $\tau_{a,b}$ if either $a$ or $b$ is even, where we assume $(a,b)=1$.

\item[$3)$]  In the case $b)$ the metric induced on $\tau_{a,b}\cong T_{a,b,\sqrt{a^2+b^2}}$ is extremal for the functionals $\Lambda_j(\mathbb{T}^2,g)$ if $\tau_{a,b}$ is a Lawson torus or $\Lambda_j(\mathbb{KL},g)$ if $\tau_{a,b}$ is a Lawson Klein bottle, where $j=2\left[\frac{\sqrt{a^2+b^2}}{2}\right]+a+b-1$ and $[\cdot]$ denotes the integer part. The corresponding value of the functional is $\Lambda_j(\tau_{a,b})=8\pi aE\left(\frac{\sqrt{a^2-b^2}}{a}\right)$.

\item[$4)$]  In the case $a)$ for an integer $k\geqslant1$ one has $T_{a,b,c}=T_{ka,kb,kc}$. Moreover, $T_{-a,b,c}$, $T_{a,-b,c}$ $T_{a,b,-c}$ and $T_{b,a,c}$ are isometric to $T_{a,b,c}$. Hence, it is sufficient to consider non-negative integer $a$, $b$, $c$ satisfying conditions $a)$ such that $(a,b,c)=1$ and assume that $(a,b,c)$ and $(b,a,c)$ are equivalent.
\end{enumerate}
\end{theorem}

The second result of the present paper concerns the maximality of metrics induced on $T_{a,b,c}$.

\begin{theorem}\label{theorem3}
All metrics induced on $T_{a,b,c}$ are not maximal except for the metric of $\tilde{\tau}_{3,1}$ and the metric of the equilateral torus.
\end{theorem}

The paper is organized in the following way. In Section~\ref{section2} we recall a construction of bipolar Lawson tau-surfaces following Lapointe's paper~\cite{lapointe}. In Section~\ref{section3} we provide a~proof of Theorem~\ref{theorem1}. Finally, nonmaximality of metrics on $T_{a,b,c}$ is shown in Section~\ref{section4}.

\section{Construction of bipolar Lawson surfaces}
\label{section2}

Let us now recall the construction of bipolar Lawson surface $\tilde{\tau}_{r,m}$ following Lapointe's paper~\cite{lapointe}. The Lawson surface $\tau_{r,m}$, with $r>m>0$ and $(r,m)=1$, is minimally immersed into $\mathbb{S}^{3}$ by $I\colon \mathbb{R}^2\rightarrow\mathbb{R}^{4}$, where
\begin{gather*}
I(u,v)=(\cos ru\cos v, \sin ru\cos v, \cos mu\sin v, \sin mu\sin v).
\end{gather*}

The bipolar minimal surface $\tilde{\tau}_{r,m}$ of $\tau_{r,m}$ is the image of an exterior product of $I$ and $I^{*}$, where $I^{*}$ is a unit vector normal to~$\tau_{r,m}$ and tangent to~$\mathbb{S}^{3}$,
\begin{gather*}
I^{*}(u,v)=\frac{(m \sin ru\sin v, -m \cos ru\sin v, -r \sin mu\cos v, r \cos mu\cos v)}{\sqrt{r^2\cos^2 v+m^2\sin^2 v}}.
\end{gather*}

The explicit formula for $\tilde{I}=I\wedge I^{*}\colon \mathbb{R}^2\rightarrow\mathbb{S}^{5}\subset\mathbb{R}^{6}$ is then
\begin{gather*}
\tilde{I}=\frac{1}{\sqrt{r^2\cos^2 v+m^2\sin^2 v}}\begin{pmatrix} -m\sin v\cos v\\ r\sin v\cos v\\ -r\cos^2 v\sin mu\cos ru-m\sin^2 v\sin ru\cos mu\\ r\cos^2 v\cos mu\sin ru+m\sin^2 v\cos ru\sin mu\\ -r\cos^2 v\sin mu\sin ru+m\sin^2 v\cos ru\cos mu\\ r\cos^2 v\cos mu\cos ru-m\sin^2 v\sin ru\sin mu \end{pmatrix}.
\end{gather*}

It is known that $\tilde{\tau}_{r,m}$ actually lies in $\mathbb{S}^{4}$, seen as an equator of~$\mathbb{S}^{5}$.

In~\cite{lapointe} Lapointe proved that for the bipolar surface $\tilde{\tau}_{r,m}$ of a Lawson torus or Klein bottle~$\tau_{r,m}$:
\begin{enumerate}\itemsep=0pt
\item[(1)] if $rm\equiv 0\!\pmod{2}$, $\tilde{\tau}_{r,m}$ is a torus with an extremal metric for $\Lambda_{4r-2}$,

\item[(2)] if $rm\equiv 1\!\pmod{4}$, $\tilde{\tau}_{r,m}$ is a torus with an extremal metric for $\Lambda_{2r-2}$,

\item[(3)] if $rm\equiv 3\!\pmod{4}$, $\tilde{\tau}_{r,m}$ is a Klein bottle with an extremal metric for $\Lambda_{r-2}$.
\end{enumerate}

The value of functional $\Lambda_{i}(\tilde{\tau}_{r,m})$ can be calculated as follows \cite{lapointe}:
\begin{enumerate}\itemsep=0pt
\item[(1)] if $rm\equiv 0\!\pmod{2}$, $\Lambda_{4r-2}(\tilde{\tau}_{r,m})=16\pi rE\left(\frac{\sqrt{r^2-m^2}}{r}\right)$,

\item[(2)] if $rm\equiv 1\!\pmod{4}$, $\Lambda_{2r-2}(\tilde{\tau}_{r,m})=8\pi rE\left(\frac{\sqrt{r^2-m^2}}{r}\right)$,

\item[(3)] if $rm\equiv 3\!\pmod{4}$, $\Lambda_{r-2}(\tilde{\tau}_{r,m})=4\pi rE\left(\frac{\sqrt{r^2-m^2}}{r}\right)$.
\end{enumerate}

\section{Proof of Theorem~\ref{theorem1}}
\label{section3}

Let us prove that the surface $T_{a,b,c}$ where $a=r-m, b=0$, and $c=r+m$ is isometric to the bipolar Lawson surface $\tilde{\tau}_{r,m}$ when $rm\equiv 0\!\pmod{2}$.

The induced metric $g$ on $T_{a,b,c}$ and $\tilde{g}$ on $\tilde{\tau}_{r,m}$ are given by the formulas~\cite{lapointe,penskoi1}
\begin{gather*}
g=\frac{1}{2}\big(c^2+\big(b^2-a^2\big)\cos 2y\big)dx^2+ \frac{c^2+\big(b^2-a^2\big)\cos 2y}{2c^2-a^2-b^2+\big(b^2-a^2\big)\cos 2y}dy^2,
\\
\tilde{g}=\frac{\big(r^2-\big(r^2-m^2\big)\sin^2 v\big)^2+r^2 m^2}{r^2-\big(r^2-m^2\big)\sin^2 v}\left(du^2+\frac{dv^2}{r^2-(r^2-m^2)\sin^2 v}\right).
\end{gather*}

The task is to f\/ind the change of variable between metrics $g$ and $\tilde{g}$. Let $a=r-m$, $b=0$, $c=r+m$ and apply the change of variable $\sin y=\operatorname{sn}(z,k)$, where $k=\frac{a}{\sqrt{a^2-c^2}}$ and $\operatorname{sn}(z,k)$ is a~Jacobi elliptic function~\cite{ellipticbook}. This implies
\begin{gather*}
g=\frac{1}{2}\big(c^2-a^2+2a^2 \operatorname{sn}^2(z,k)\big)\left(dx^2+ \frac{dz^2}{c^2-a^2}\right).
\end{gather*}
Use the following transformation
\begin{gather}
\label{transform1}
H_{1}(x,z)=\left(u,\frac{2\sqrt{c^2-a^2}}{a+c}w+K(k)\right).
\end{gather}
This implies that
\begin{gather*}
g=\left(2rm+(r-m)^2\frac{1-\operatorname{sn}^2\big(2\sqrt{\frac{m}{r}}w,k\big)}{1-k^2\operatorname{sn}^2\big(2\sqrt{\frac{m}{r}}w,k\big)}\right)
\left(du^2+\frac{dw^2}{r^2}\right)
\\
\hphantom{g=} =\frac{\big(r^2-\big(r^2-m^2\big)\operatorname{sn}^2\big(w,\frac{\sqrt{r^2-m^2}}{r}\big)\big)^2+(rm)^2}{r^2-\big(r^2-m^2\big)
\operatorname{sn}^2\big(w,\frac{\sqrt{r^2-m^2}}{r}\big)} \left(du^2+\frac{dw^2}{r^2}\right),
\end{gather*}
which for $k'=\sqrt{1-k^2}$, used the identities~(13.17), (13.22), (13.23) from~\cite{ellipticbook}:
\begin{alignat*}{3}
& \operatorname{cn}^2(t,k)=1-\operatorname{sn}^2(t,k) ,\qquad &&
\operatorname{sn}\left(k'u,\frac{ik}{k'}\right)=k'\frac{\operatorname{sn}(u,k)}{\operatorname{dn}(u,k)},&\\
& \operatorname{dn}^2(t,k)=1-k^2\operatorname{sn}^2(t,k) ,\qquad && \operatorname{sn}\left((1+k')u,\frac{1-k'}{1+k'}\right)=(1+k')\frac{\operatorname{sn}(u,k) \operatorname{cn}(u,k)}{\operatorname{dn}(u,k)} .&
\end{alignat*}

Now apply change of variable $\sin v=\operatorname{sn}(w,\tilde{k})$, where $\tilde{k}=\frac{\sqrt{r^2-m^2}}{r}$, implying
\begin{gather*}
g=\frac{\big(r^2-\big(r^2-m^2\big)\sin^2 v\big)^2+r^2 m^2}{r^2-\big(r^2-m^2\big)\sin^2 v}\left(du^2+\frac{dv^2}{r^2-\big(r^2-m^2\big)\sin^2 v}\right)=\tilde{g}.
\end{gather*}

When $rm\equiv 0\!\pmod{2}$, $a=r-m$ and $c=r+m$ are both odd since $(r,m)=1$. We have that $T_{a,b,c}$ is a torus and $\tilde{F}_{a,b,c}\colon \mathbb{R}^2/\mathcal{L}\rightarrow T_{a,b,c}$, where $\mathcal{L}=\{(2\pi n,2\pi m)\,|\,n,m\in \mathbb{Z}\}$, is a one-to-one map~\cite{penskoi1}. Apply change of variable $\sin y=\operatorname{sn}(z,k)$, and we alternatively have $\hat{F}_{a,b,c}\colon \mathbb{R}^2/\mathcal{\hat{L}}\rightarrow T_{a,b,c}$, where $\mathcal{\hat{L}}$ $=\{(2n\pi,4mK(k))\,|\,n,m\in \mathbb{Z}\}$. There is now a one-to-one correspondence between the rectangle~$[0,2\pi)\times [K(k),5K(k))$ and $T_{a,b,c}$. Our linear transformation~\eqref{transform1} maps this rectangular domain as follows {\samepage
\begin{gather*}
 H_{1}\left([0,2\pi)\times [K(k),5K(k))\right)=[0,2\pi)\times \left[0,\frac{2(a+c)K(k)}{\sqrt{c^2-a^2}}\right)=[0,2\pi)\times \big[0,2K\big(\tilde{k}\big)\big),
\end{gather*}
which used the identities $K(k)=\frac{1}{k'}K(\frac{i k}{k'})$, $K(\tilde{k})=\frac{2}{1+\tilde{k}'}K\big(\frac{1-\tilde{k}'}{1+\tilde{k}'}\big)$. }

Let us remark that when $rm\equiv 0\!\pmod{2}$ and after change of variable $\sin v=\operatorname{sn}(w,\tilde{k})$, the bipolar Lawson torus $\tilde{\tau}_{r,m}$ has a one-to-one correspondence with $[0,2\pi)\times [0,2K(\tilde{k}))$.

Thus we obtained the required isometry.

We recall now the explicit def\/inition of $S(a,b,c)$ \cite{penskoi1}:
\begin{gather}
\label{Sabc}
S(a,b,c)=\frac{4\pi}{\sqrt{c^2-a^2}}\left(2\big(c^2-a^2\big)E\left(\sqrt{\frac{b^2-a^2}{c^2-a^2}}\right)-
\big(c^2-a^2-b^2\big)K\left(\sqrt{\frac{b^2-a^2}{c^2-a^2}}\right)\right),
\end{gather}
where $K(k)$, $E(k)$ are the complete elliptic integrals of the f\/irst, second kind~\cite{ellipticbook}.

Using the values $\Lambda_{2(a+c)-2}(T_{a,b,c})$ and $\Lambda_{4r-2}(\tilde{\tau}_{r,m})$ from~\cite{penskoi1} and~\cite{lapointe}, which are twice the areas of these surfaces, we can check that the areas of corresponding surfaces are the same. Note that $S(a,b,c)=S(b,a,c)$ since $T_{b,a,c}\cong T_{a,b,c}$,
\begin{gather*}
\Lambda_{2(a+c)-2}(T_{a,0,c})=2S(0,a,c)=\frac{8\pi}{c}\left[2c^2 E\left(\frac{a}{c}\right)-\big(c^2-a^2\big)K\left(\frac{a}{c}\right)\right] \\
\hphantom{\Lambda_{2(a+c)-2}(T_{a,0,c})}{}
=8\pi(a+c) E\left(\frac{2\sqrt{ac}}{a+c}\right)=\Lambda_{4r-2}(\tilde{\tau}_{r,m}),
\end{gather*}
which uses the identity $E\big(\frac{2\sqrt{k}}{1+k}\big)=\frac{2E(k)-(1-k^2)K(k)}{1+k}$.

We now consider the cases when $rm\equiv 1\!\pmod{4}$ and $rm\equiv 3\!\pmod{4}$.
For the case of $rm\equiv 1\!\pmod{4}$, we have that $\tilde{\tau}_{r,m}$ is a bipolar Lawson torus, and it is isometric to the surfa\-ce~$T_{a,b,c}$, where $a=\frac{r-m}{2}$, $b=0$, and $c=\frac{r+m}{2}$.

Remark that the change of variable on $y$ and $v$ are the same as in the case $rm\equiv 0\!\pmod{2}$. The method of proof is the same using the transformation
\begin{equation}
\label{transform2}
H_{2}(x,z)=\left(2u,\frac{2\sqrt{c^2-a^2}}{a+c}w+K(k)\right).
\end{equation}

When $rm\equiv 1\!\pmod{4}$, $a=\frac{r-m}{2}$ is even and $c=\frac{r+m}{2}$ is odd. We have that $T_{a,b,c}$ is again a torus and $\tilde{F}_{a,b,c}\colon \mathbb{R}^2/\mathcal{L}\rightarrow T_{a,b,c}$ is a one-to-one map~\cite{penskoi1}. Let us again use the rectangle $[0,2\pi)\times [K(k),5K(k))$ and $\hat{F}_{a,b,c}\colon \mathbb{R}^2/\mathcal{\hat{L}}\rightarrow T_{a,b,c}$, where $\mathcal{\hat{L}}$ $=\{(2n\pi,4mK(k))\,|\,n,m\in \mathbb{Z}\}$. Now, our linear transformation~\eqref{transform2} maps this rectangular domain as follows
\begin{gather*}
H_{2}\big([0,2\pi)\times [K(k),5K(k))\big)=[0,\pi)\times \left[0,\frac{2(a+c)K(k)}{\sqrt{a^2-c^2}}\right)=[0,\pi)\times \big[0,2K\big(\tilde{k}\big)\big) .
\end{gather*}

When $rm\equiv 1\!\pmod{4}$ and after change of variable $\sin v=\operatorname{sn}(w,\tilde{k})$, the bipolar Lawson torus $\tilde{\tau}_{r,m}$ has a~one-to-one correspondence with $[0,\pi)\times [0,2K(\tilde{k}))$.

Thus we obtained the required isometry.

Now, for the case of $rm\equiv 3\!\pmod{4}$, we have that $\tilde{\tau}_{r,m}$ is a bipolar Lawson Klein bottle, and it is isometric to $T_{a,b,c}$, where $a=\frac{r-m}{2}$, $b=0$, and $c=\frac{r+m}{2}$.

Remark that we need the same change of variable on $y$ and $v$ and transformation~\eqref{transform2} as the case $rm\equiv 1\!\pmod{4}$.

When $rm\equiv 3\!\pmod{4}$, $a=\frac{r-m}{2}$ is odd and $c=\frac{r+m}{2}$ is even. We have that $T_{a,b,c}$ is a Klein bottle and $\tilde{F}_{a,b,c}\colon \mathbb{R}^2/\mathcal{L}\rightarrow T_{a,b,c}$ is a double covering~\cite{penskoi1}. With change of variable $\sin y=\operatorname{sn}(z,k)$ we now have as before a one-to-one correspondence between the rectangle $[0,\pi)\times [K(k),5K(k))$ and $T_{a,b,c}$. Our linear transformation~\eqref{transform2} maps this rectangular domain as follows
\begin{gather*}
H_{2}\big([0,\pi)\times [K(k),5K(k))\big)=\left[0,\frac{\pi}{2}\right)\times \left[0,\frac{2(a+c)K(k)}{\sqrt{a^2-c^2}}\right)
=\left[0,\frac{\pi}{2}\right)\times \big[0,2K\big(\tilde{k}\big)\big).
\end{gather*}

When $rm\equiv 3\!\pmod{4}$ and after change of variable $\sin v=\operatorname{sn}(w,\tilde{k})$, the bipolar Lawson Klein bottle $\tilde{\tau}_{r,m}$ has a one-to-one correspondence with $\big[0,\frac{\pi}{2}\big)\times [0,2K(\tilde{k}))$.

This completes the proof of Theorem~\ref{theorem1}.

\section{Proof of Theorem~\ref{theorem3}}
\label{section4}

In~\cite{kar3}, Karpukhin studied nonmaximality of known extremal metrics including metrics on Otsuki tori $O_{\frac{p}{q}}$, Lawson tori and Klein bottles~$\tau_{r,m}$, generalized tori~$M_{r,m}$, bipolar Lawson surfaces~$\tilde{\tau}_{r,m}$, and bipolar Otsuki tori~$\tilde{O}_{\frac{p}{q}}$. Although the family $T_{a,b,c}$ of tori and Klein bottles include the families $\tau_{r,m}$, $M_{r,m}$ and now $\tilde{\tau}_{r,m}$, there has been no general study of nonmaximality of metrics on~$T_{a,b,c}$. We prove that there are no maximal metrics among $T_{a,b,c}$ apart from the equilateral torus $T_{1,1,2}\cong M_{1,1}$ and Klein bottle $T_{1,0,2}\cong\tilde{\tau}_{3,1}$.

\begin{proposition}[Karpukhin, \cite{kar3}]\label{proposition1} The following inequalities hold
\begin{gather*}
\sup \Lambda_{n}\big(\mathbb{T}^2,g\big)\geq 8\pi\left(n-1+\frac{\pi}{\sqrt{3}}\right), \\
\sup \Lambda_{n}(\mathbb{KL},g)\geq 8\pi(n-1)+12\pi E\left(\frac{2\sqrt{2}}{3}\right),
\end{gather*}
where $E(k)$ is the complete elliptic integral of the second kind~{\rm \cite{ellipticbook}}.
\end{proposition}

This result can also be found by Colbois and El Souf\/i in~\cite[Theorem~B:4]{colelso}. Recent computational evidence~\cite{compu} further suggests that this construction is in fact maximal.  For our purposes, we use Proposition~\ref{proposition1} to show nonmaximality on~$T_{a,b,c}$.

\begin{proposition}\label{proposition2}
Let $a b c \ne 0$, and $|c|>\sqrt{a^2+b^2}$. If $a$ and $b$ have different parity and~$c$ is even, then the following inequality holds
\begin{gather*}
\Lambda_{a+b+c-3}(T_{a,b,c})<8\pi(a+b+c-4)+12\pi E\left(\frac{2\sqrt{2}}{3}\right).
\end{gather*}
If $a$ and $b$ are odd and $c$ is even, then the following inequality holds
\begin{gather*}
\Lambda_{a+b+c-3}(T_{a,b,c})<8\pi\left(a+b+c-4+\frac{\pi}{\sqrt{3}}\right).
\end{gather*}
Otherwise, the following inequality holds
\begin{gather*}
\Lambda_{2(a+b+c)-3}(T_{a,b,c})<8\pi\left(2(a+b+c)-4+\frac{\pi}{\sqrt{3}}\right).
\end{gather*}
\end{proposition}

Before proving Proposition~\ref{proposition2}, let us f\/ind an upper bound for $S(a,b,c)$.

\begin{proposition}\label{proposition3}
Let $|c|>\sqrt{a^2+b^2}$. The following inequality holds
\begin{gather*}
S(a,b,c)<2\pi^2(a+b+c).
\end{gather*}
\end{proposition}

\begin{proof} Using the def\/inition \eqref{Sabc} of $S(a,b,c)$, we show the proof when $a$, $b$, $c$ are each nonzero. The instance when $a,b,c$ contains at least one zero will follow as a special case. Here, we now use $a,b,c>0$ without loss of generality. Since $c>\sqrt{a^2+b^2}$ we have that our interval on~$k$ is~$[0,1)$ for both $K(k)$ and $E(k)$ in the inequality above. It is well known \cite{ellipticbook} that $K(k)$ is bounded below by~$\frac{\pi}{2}$, and $E(k)$ is bounded above by~$\frac{\pi}{2}$ on the interval~$[0,1]$. This implies that
\begin{gather*}
S(a,b,c)\le\frac{4\pi}{\sqrt{c^2-a^2}}\left(2\big(c^2-a^2\big)\frac{\pi}{2}-\big(c^2-a^2-b^2\big)\frac{\pi}{2}\right).
\end{gather*}
We now simplify this expression as follows
\begin{gather*}
S(a,b,c)\le 2\pi^2\sqrt{c^2-a^2}+\frac{2b^2\pi^2}{\sqrt{c^2-a^2}}.
\end{gather*}
We have that $\sqrt{c^2-a^2}\le\sqrt{c^2+a^2}\le\sqrt{c^2+2c a+a^2}=\sqrt{(c+a)^2}=(c+a)$. Hence, this implies that
\begin{gather*}
S(a,b,c)\le 2\pi^2(c+a)+\frac{2b^2\pi^2}{\sqrt{c^2-a^2}}.
\end{gather*}
The condition $c>\sqrt{a^2+b^2}$ implies that $\frac{b}{\sqrt{c^2-a^2}}<1$. This now implies
\begin{gather*}
S(a,b,c)<2\pi^2(c+a)+2\pi^2 b.
\end{gather*}
This concludes the proof of Proposition~\ref{proposition3}.
\end{proof}

We recall that Lawson tau-surfaces $\tau_{r,m}$ correspond to $T_{a,b,\sqrt{a^2+b^2}}$. Since $c=\sqrt{a^2+b^2}$, we have the same upper bound for $S(a,b,c)$ with the exception that the inequality is not strict.

\begin{proof}[Proof of Proposition~\ref{proposition2}]
In order to prove the f\/irst inequality we must show
\begin{gather*}
0<8\pi(a+b+c-4)+12\pi E\left(\frac{2\sqrt{2}}{3}\right)-S(a,b,c).
\end{gather*}
Using Proposition~\ref{proposition3}, this implies that we must show
\begin{gather*}
0<8\pi(a+b+c-4)+12\pi E\left(\frac{2\sqrt{2}}{3}\right)-2\pi^2(a+b+c).
\end{gather*}
This expression can be rewritten as
\begin{gather*}
\frac{16-6 E\left(\frac{2\sqrt{2}}{3}\right)}{4-\pi}<(a+b+c).
\end{gather*}
This inequality holds for $a+b+c\geq 11$. For the exceptional cases $(a,b,c)\in\{(1,2,4),(1,2,6)$, $(2,3,4)\}$, one can verify the f\/irst inequality explicitly using the tables of elliptic integrals in the book~\cite{integraltables}.

Similarily, the second inequality holds for $a+b+c\geq 11$. Remark that when $(a,b,c)=(1,1,2)$, we have the well known equilateral torus $T_{1,1,2}\cong M_{1,1}$. For the exceptional cases $(a,b,c)\in\{(1,1,4),(1,1,6),(1,1,8),(1,3,4),(1,3,6)$, $(3,3,4)\}$, one can again verify the second inequality explicitly.

The third inequality holds for $a+b+c\geq 6$. Only one exceptional case of $(a,b,c)=(1,1,3)$ should be verif\/ied for the third inequality explicitly.
\end{proof}

We now observe when exactly one of $a$ or $b$ is zero, while $|c|>\sqrt{a^2+b^2}$. Since $T_{a,b,c}\cong T_{b,a,c}$ when we assume that $S(a,b,c)=S(b,a,c)$, this corresponds to surfaces $T_{a,0,c}$. These surfaces have now been shown isomorphic to the family of bipolar Lawson surfaces $\tilde{\tau}_{r,m}$. Fortunately, these surfaces were investigated by Karpukhin in~\cite{kar3} and nonmaximality was shown for all extremal metrics except the well known bipolar Lawson Klein bottle $T_{1,0,2}\cong \tilde{\tau}_{3,1}$.

In~\cite{penskoi1}, the case of $T_{0,0,1}$ was already investigated to be the Clif\/ford torus but with a metric multiplied by $\frac{1}{2}$. Remark that the condition of $(a,b,c)=1$ prevents any other~$T_{0,0,n}$.

Finally, when $c=\sqrt{a^2+b^2}$, we have seen in Theorem~\ref{theorem2} that this corresponds to Lawson tau-surfaces $\tau_{r,m}\cong T_{a,b,\sqrt{a^2+b^2}}$. These surfaces were also investigated by Karpukhin in~\cite{kar3}, and it was shown that none of the extremal metrics are maximal.

This completes the proof of Theorem~\ref{theorem3}.

\subsection*{Acknowledgements}

This result was obtained during studies of the author at the National Research University~-- Higher School of Economics, Moscow and the author is very grateful for its hospitality.  The author also thanks A.V.~Penskoi for the statement of this problem, many useful discussions and invaluable help in preparing this manuscript.
The research of the author was partially supported by an NSERC Postgraduate Fellowship and by AG Laboratory NRU-HSE, Russian Federation government grant, ag.~11.G34.31.0023.

\pdfbookmark[1]{References}{ref}
\LastPageEnding

\end{document}